\documentclass[final,3p,times]{elsarticle}

\usepackage[french]{varioref}
\usepackage{amssymb,amsfonts,amsmath,amsthm}
\usepackage{wrapfig}
\usepackage{inputenc}
\usepackage[T1]{fontenc}
\usepackage[english]{babel}
\newcommand{\prs}{\langle\;,\;\rangle}
\newcommand{\too}{\longrightarrow}

\newcommand{\esp}{\quad\mbox{and}\quad}

\newcommand{\X}{{\cal X}}
\newcommand{\G}{{\mathfrak{g}}}

\newcommand{\h}{{\mathfrak{h}}}

\newcommand{\ric}{{\mathrm{ric}}}
\newcommand{\Ri}{{\mathrm{Ric}}}
\newcommand{\Ku}{{\mathrm{K}}}

\newcommand{\B}{{\cal B}}

\newcommand{\na}{\nabla}

\newcommand{\al}{\alpha}

\newcommand{\we}{\wedge}
\newcommand{\ga}{\gamma}

\newcommand{\e}{\epsilon}

\newcommand{\la}{\lambda}

\newcommand{\val}{\mapsto}

\newcommand{\vals}{\rightarrow}

\newcommand{\ssi}{\Longleftrightarrow}

\newcommand{\tr}{\mathfrak{trace}}
\font\bb=msbm10

\def\K{\hbox{\bb K}}

\def\B{\hbox{\bb B}}
\def\R{\hbox{\bb R}}

\def\N{\hbox{\bb N}}

\newtheorem{Def}{Definition}[section]
\newtheorem{theo}{Theorem}[section]
\newtheorem{pr}{Proposition}[section]
\newtheorem{Le}{Lemma}[section]
\newtheorem{co}{Corollary}[section]

\newtheorem{exem}{Example}
\newtheorem{rem}{Remark}

\begin{document}

\begin{frontmatter}

\title{ On the semi-symmetric Lorentzian spaces.}
 \author{Abderrazzak Benroummane}
 \address{Department of Computer Engineering and Mathematics\\National School of Applied Sciences of Berrechid \\Hassan 1st University, Morocco.\\e-mail: abderrazzak.benroummane@uhp.ac.ma
 }



\begin{abstract}
	We give some properties of semisymmetric pseudo-Riemannian manifolds. These are foliated manifolds and for the Lorentzian metric, Ricci operator has only real eigenvalues.
\end{abstract}

\begin{keyword} Lorentzean manifolds \sep  Semi-symmetric pseudo-Riemannian manifolds \sep  \sep

\end{keyword}

\end{frontmatter} 

\section{Introduction}

A pseudo-Riemannian manifold $ (M, g) $ is said  globally symmetric (respectively, locally symmetric ) if any point $ m \in M $ is a fixed point of a non-trivial involutive isometric $ \ga_m $ (respectively, any point $ m \in M $ admits a neighborhood $ N_m $ and the symmetric geodesic map $ s_m: N_m \vals N_m $ is an isometry, where $s_m$ is defined by $ x \val s_m (x) = \ga_x (-1) $ such that  $ \ga_x $ is the geodesic curve from $ m = \ga_x (0) $ to $ x = \ga_x (1) $ for all $ x \in N_m $)  [See \cite{besse}].

It is necessary that a globally symmetrical manifold is locally symmetric but is not sufficient unless the manifold is simply connected. So, that a manifold $ (M, g) $ is locally symmetric if and only if its Riemann curvature tensor $ R $ is parallel, that is to say that $ \na R = 0 $ where $\na$ is the Levi-Civita connection associeted to metric $g$. Indeed; the necessary condition is obvious from the fact that $ (ds_m) _m = -Id_ {T_mM} $ and for the sufficient condition it is a result of the following theorem:
\begin{theo}\cite{besse}
Let $ (M, g) $ be a pseudo-Riemannian manifold. Let $ \na $ and $ R $ be the Levi Civita connection and the Riemann curvature tensor associated to the metric $ g $ respectively.

Let $ (x, y) \in M^2 $ and $ \tau: T_xM \vals T_yM $ be an isomorphism. Suppose that $ \tau.R_x = R_y $ and $ \na R = 0 $. Then, there exists a neighborhood $ V_x $ of $ x $ and a diffeomorphism $ \varphi: V_x \vals V_y $ with values in a neighborhood $ V_y $ of $ y $ such that $ (d \varphi) _x = \tau $ .
\end{theo}    

E. Cartan in (\cite{cartan, Cartan1}) showed that a locally symmetric Riemannian manifold $ (M, g) $ is locally homogeneous, more precisely, $ M = G / H_0 $ where $ G $ is the connected component of $ id_M $ in $ Iso(M) $ the Lie group of isometries of $ M $ such that  $ G$ acts transitively on $ M $. $ H_0 $ is the isotropic subgroup at the point $ m \in M $. This structure is independent of the chosen point $ m $.
  
  On the other hand, $ \G: = {\h}_0 + T_mM $ is the Lie algebra of $ G $,  where $\h_0 $ is the Lie algebra of $ H_0 $, $ T_mM $ is the tangent space of $ M $ at the point $ m $. The Lie bracket on $ \G $ is given by:
\begin{equation}\label{equa1}
\begin{array}{cccc}	\left[ X,Y \right]  &:=&R_m(X,Y),&\;\; \forall X,\; Y\in T_mM,\\
		\left[ A,X\right] &:=&A(X),&\;\; \forall X\in T_mM,\; \forall A\in\h_0,\\
			 \left[ A,B\right] &:=&AB-BA,&\;\; \forall A,\; B\in \h_0.\end{array}\end{equation}
			 
			 Moreover, $R$  checked \begin{equation}\label{1equa2}
			 \left[ R_m(x,y),R_m(z,t)\right] =R_m(R_m(x,y)z,t)+R_m(z,R_m(x,y)t),\; \forall (x,y,z,t)\in (T_mM)^4,
			 \end{equation}
\begin{Def}\begin{enumerate}
\item[a)]  Let $(V,g)$ be a pseudo-Riemannian vector space, 
	 $ \h $ be a Lie subalgebra of $so(V)$ and $ R $ be a curvature tensor on $ V $. If $ R (x, y) \in \h $ for all $ (x, y) \in V^2 $, the triple $ \left [V, \h, R \right] $ is said  holonomy sysmtem. If moreover $R$ satisfies the relationship
	 \begin{equation}\label{eqsim12}
	 (A.R)(x,y) =[A,R(x,y)]-R(A(x),y)-R(x,A(y)=0,\; \forall (x,y)\in V^2,
	 \end{equation}
	 the holonomy sysmtem $ \left [V, \h, R \right] $ is said to be semmetric  (see\cite{simon}). 
\item[b)]
	 A pseudo-Riemannian manifold $ (M, g) $ is said to be semi-symmetric if, its Riemann curvature tensor satisfies the equation (\ref {1equa2}). 
\end{enumerate}

\end{Def}
 
Note that if, $ (M, g) $ is a semi-symmetric pseudo-Riemannian manifold, then, at any point $m\in M$, the triple $ \left [T_mM, \h(R_m), R_m \right] $ is symmetric, where $ \h(R_m)$ is the vector space spanned by all $R(x,y)$ for all  $x$ and $y$ in $T_mM$, moreover $ \h(R_m)$ is  a Lie algebra.  This result allows K. Nomizu in 1968 (\cite{nomizu}) to conjecture that any semi-symmetric Riemannian manifold $ (M, g) $  of dimension greater than or equal to $ 3 $ is locally symmetric (ie, $ \na R = 0 $). But, in 1972, H. Takagi (\cite{takagi}) gave a counter example of a Riemannian manifold of dimension $ 3 $ satisfying $ R(X, Y).R = 0 $ and not $ \na R = 0 $.

On the other hand, semi-symmetry is a generalization of two-symmetry ($ \na^2R = 0 $) and of local symmetry. Several studies are done on semi-symmetry. In the Riemannian case, Z. I Szabo gave the complete classification of semi-symmetric Riemannian manifolds (\cite{zabo, zabo1, zabo2}). For the other strictly pseudo-Riemannian metrics, there is no general study, the few cases studied are with additional conditions.   We can mention   the complete classification of four dimensional semisymmetric homogeneous Lorentzian manifolds given by the auther , M. boucetta and A. Ikemakhen (See \cite{benromane}). In    \cite{benromane2}, the auther gave the complete classification of four dimensional semisymmetric homogeneous neutral manifolds. In \cite{calvaruso1}, G. Calvaruso   has studied  the three-dimensional semi-symmetric homogeneous Lorentzian manolds. G. Calvaruso and  B. De Leo   have studied the  semi-symmetric Lorentzian three-manifolds admitting a parallel degenerate line field (See \cite{calvaruso2}). In \cite{ali}, A. Haji-Badali and A. Zaeim gave a complete classification of Semi-symmetric four dimensional neutral Lie groups.

The purpose of this work is to give a few characterizations of semi-symmetric Lorentzian manifolds where we give the following results:
\begin{enumerate}
	\item Any semi-symmetric Lorentzian manifold, its Ricci operator has only real eigenvalues.
	\item A semi-symmetric pseudo-Riemannian manifold $ (M, g) $ is   foliated manifold. In the case where the metric is Lorentzian, the restriction of the Ricci operator on each leaf has at most one non-zero real  eigenvalue.
	\item If $ (M, g) $ is a simple leaf semi-symmetric Lorentzian manifold, then the Ricci opertor is diagonalizable or isotropic. Moreover, if  $ \la $ is a non-zero  eigenvalue of Ricci and $ V_1: = \ker (\Ri- \la id_{TM}) $ is Riemannian subspace of dimonsion $\geq 3$, then,   $ 1 \leq \dim (\ker (\Ri)) \leq 2 $.
\end{enumerate} 
  This paper is organized in the following way.
  In the second paragraph, we give generalities on curvature tensors on a pseudo-Riemannian vector space $ V $. Thus, characterizations of a semi-symmetric curvature tensor by decomposing the space $ V $ as a direct sum of the characteristic subspaces of the Ricci operator. We also give another so-called primitive decomposition under the action of primitive holonomy algebra. In third paragraph, we give some proprties of \emph{Ricci decomposition} and  the \emph{primitive decomposition} of  tangent bundle of the  semi-symmetric pseudo-Riemannian manifold.  The last paragph   is devoted to studying  semi-symmetric Lorentzian manifolds  which  is a simple leaf  where Ricci is diagonalizabe or  isotropic  and if Ricci admits a non-zero eigenvalue then it is diagonalizable. Moreover,  if this eigenvalue has a multiplicity hose greater than or equal to $ 3 $ with Riemannian eigenspace, we will have $ 1 \leq \dim (\ker (\Ri)) \leq 2 $.

 At first, we  give some definitions that we will use later
\begin{Def}
	Let $ (M, g) $ be a locally connected pseudo-Riemannian manifold.\begin{enumerate}\item $ (M, g) $ is said to be
		irreducible if, at each point $m\in M$, the only subspace of $ T_mM $ which are invariant under  the action of the holonomy group $ H_m (M) $ are $ \{0 \} $ and $ T_mM $, and reducible otherwise.
		\item $ (M, g) $ is said to be indecomposable if, at any point $ m \in M $, the only nondegenerate subspaces of $ T_mM $ which are invariant under  the action of the holonomy Lie group $ H_m (M) $  are $ \{0 \} $ and $ T_mM $.
		\item $ (M, g) $ is said to be weakly irreducible if it is reducible and indecomposable.
	\end{enumerate}
\end{Def}
\section{Semisymmetrical curvature tensor on a pseudo-Riemannian vector space }
\subsection{Curvature tensor}
Let $(V,\prs)$ be a $n$-dimensional pseudo-Riemannian vector space. We identify $V$ and  its dual $V^*$ by the means of $\prs$. This implies that the Lie algebra $V\otimes V^*$ of endomorphisms of $V$ is identified with $V\otimes V$, the Lie algebra $\mathrm{so}(V,\prs)$ of skew-symmetric endomorphisms is identified with $V\wedge V$ and the space of symmetric endomorphisms is identified with $V\vee V$ (the symbol $\wedge$ is the outer product and $\vee$ is the symmetric product). For any $u,v\in V$,
\[ (u\wedge v)w=\langle v,w\rangle u-\langle u,w\rangle v\esp (u\vee v)w=\frac12\left(\langle v,w\rangle u+\langle u,w\rangle v\right). \]

On the other hand, $V\wedge V$ carries  also a nondegenerate symmetric product also denoted by $\prs$ and given by
\[ \langle u\wedge v,w\wedge t\rangle:=\langle u\wedge v(w), t\rangle=
\langle v,w\rangle\langle u,t\rangle-\langle u,w\rangle\langle v,t\rangle. \]
We identify $V\wedge V$ with its dual by means of this metric.


We consider  Bianchi's linear mapping of   the space $ P =\vee^2 (\we^ 2V)  $ given by:
   \begin{equation}\label{eq1-4}B\bigg( (a\we b)\vee (c\we d)\bigg) =(a\we b)\vee (c\we d)+(b\we c)\vee (a\we d)+(c\we a)\vee (b\we d).\end{equation}
  Let     $ \G $ be a subalgebra of $so(V)$, we set:
  \[ R(\G):=\ker(B)=\{T \in \G\vee \G /~~ B(T)=0\} \]et
    \[\G_{sym}=\{T\in R(\G) / \G.T=0\}.\]
    
  The set  $R(\G)$  is called the  space of all {\it curvature tensors} of type  $\G$ and any element of $ R (so (V)) $ is called {\it curvature tensor} on $ V $. 
The set $ \G_ {sym} $ is called the space of symmetric curvature tensors of type $ \G $.

    According identifications cited above, we obtained;

       \begin{Le}\label{le1-1}  Let $ (V, \prs) $ be a pseudo-Riemannian vector space. Any curvature tensor $ \Ku \in R (so (V)) $ can be identified with an element of $ \otimes^4V^* $ (i.e., a covariant 4-tensor on V) satisfying:
\begin{enumerate}\item[i)] $\Ku(a,b,u,v)=-\Ku(b,a,u,v)$,
\item[ii)] $\Ku(a,b,u,v)=-\Ku(a,b,v,u)$,
\item[iii)] $\Ku(a,b,u,v)+\Ku(b,u,a,v)+ \Ku(u,a,b,v)=0,$
\item[iv)] $\Ku(a,b,u,v)=\Ku(u,v,a,b)$,
\end{enumerate}where  $a,b,u,v \in V$. Note that $ iv) $ is a result of $i)$, $ii)$ et $iii)$.

\end{Le}

\begin{Le}\label{1-2} Let $ (V, \prs) $ be a pseudo-Riemannian vector space. Any curvature tensor $ \Ku \in R (so (V)) $ can be identified with a symmetric bilinear map also denoted by $ \Ku: V \we V \too V \we V $ satisfying:
\begin{enumerate}\item for all  $u$, $v\in V$, $\Ku(u\we v)=-\Ku(v\we u)$,
\item for all  $u,v,w\in V$, $\Ku(u\we v)w+\Ku(v\we w)u+\Ku(w\we u)v=0$.
\end{enumerate} These relationships lead to:
\begin{equation}\label{eq1-5}
\langle \Ku(a\we b)u,v\rangle=\langle \Ku(u\we v)a,b\rangle,\quad a,b,u,v\in V.
\end{equation}We often set: $$\Ku(u,v):=\Ku(u\we v)$$
\end{Le}

 Let $ (V, \prs) $ be a pseudo-Riemannian vector space and let $\Ku\in R(so(V))$ be a curvature tensor  $V$.  We  set:  \[\h(\Ku):=\Ku(\we^2V)=\{\Ku(u,v)/\;u,\;v\in V\},\]
 $\h(\Ku)$ is a vector subspace of $so(V)$. We set $\bar{\h}(\Ku)$ the smallest Lie subalgebra of $so(V)$ containing $\h(\Ku)$  called {\it primitive holonomy algebra of $\Ku$}, its  Lie groupe is noted     $\bar{\cal H}(\Ku)$  and  called  {\it  primitive holonomy group}.  The action of $\bar{\h}(\Ku)$  and  $\bar{H}(\Ku)$) on the curvature tensor $\Ku$ are given respectively by:
   \begin{equation}\label{eq1-6}
    (A.\Ku)( a,b):= [A,\Ku(a,b)]- \Ku(A(a),b)-\Ku(a,A(b)),
  \end{equation} and
   \begin{equation}\label{eq1-6"}
    (\sigma.\Ku)( a,b):= \sigma\circ\Ku(\sigma^{-1}( a), \sigma^{-1}(b))\circ \sigma^{-1} ,
  \end{equation}

    where $A\in \bar{\h}(\Ku)$, $\sigma \in \bar{H}(\Ku)$ and  $a$, $b\in V$.

The Ricci curvature associated to $\mathrm{K}$ is the symmetric bilinear form on $V$ given by $\mathfrak{ric}_{\Ku}(u,v)= trace(\tau(u,v))$, where $\tau(u,v):V\too V$ is given by $\tau(u,v)(a)=\mathrm{K}(u, a)v$. The  Ricci operator is the symmetric endomorphism $\mathrm{Ric}_K:V\too V$ given by $\langle
\mathrm{Ric}_{\Ku}(u),v\rangle =\mathfrak{ric}_{\Ku}(u,v), \,\,\text{ for all } u,v\in V$.   We call $\Ku$ Einstein (resp. Ricci isotropic) if $\Ri_{\Ku}=\la\mathrm{Id}_V$ (resp. $\Ri_{\Ku}\not=0$ and $\Ri_{\Ku}^2=0$).

\begin{exem}  if ${\Ku}=(u\wedge v)\vee(w\wedge t)$ then,
\[ \ric_K=\langle u,w\rangle t\vee v+\langle v,t\rangle u\vee w-\langle v,w\rangle t\vee u-
\langle u,t\rangle v\vee w.\]
\end{exem}
   Note that, if there is no ambiguity, we put $\Ri=\Ri_{\Ku}$ and $\mathfrak{ric}=\mathfrak{ric}_{\Ku}$.


 \subsection{Semi-symmetrical curvature tensors}
 \subsubsection{Primitive decompotion}
 
 \begin{Def}Let $(V,\prs)$ be a $n$-dimensional pseudo-Riemannian vector space and let $\Ku$ a curvature tensor on $ V $.
 	The curvature tensor $ \Ku $ is called {\it semi-symmetric} if $ \h(
 	\Ku) $ is a Lie algebra and $ \Ku $ is a symmetric curvature tensor of type $\h (\Ku)$, i.e. $ \Ku \in \h(\Ku)_{sym}. $
 \end{Def}
 
  \begin{rem}
 	
 	In a pseudo-Riemannian space $ (V,\prs) $ of small dimension, to give the semi-symmetric curvature tensors on $ V $, it suffices to give the Lie subalgebras $ \G $ of $ so (V) $ satisfying, $ \G = \G_ {sym} $ (For example, see \cite{benromane, benromane2}).
 \end{rem}
 
 It is obvious that if the curvature $\Ku$ is  constant,  i.e.   $\Ku=\la Id_{V\wedge V}$, then it is   semi-symmetric  and if $\la\neq 0$, we get   $\mathfrak{h} (\Ku)=\mathrm{so}(V)$. In particular, any   curvature tensor on the space pseudo-riemannian of dimension $2$ is semi-symmetric.
 
  \begin{pr} Let $(V,\prs)$ be a  pseudo-Riemannian vector space of dimension $n$ and let $\Ku$ be a curvature tensor on $ V $ treated as an
  	symmetrical endomorphism,   
     $\Ku\;:\:a\we b \in \we^2 V\val \Ku(a\we b)\in \we^2V$. Then,
     
    $\Ku$ is   semi-symmetric iff  $\Ku.\Ku=0$, i.e.  $\Ku$ checks; \begin{equation}\label{eq1-9}
 [\Ku(u,v),\Ku(a,b)]=\Ku(\Ku(u,v)a,b)+\Ku(a,\Ku(u,v)b),\quad u,v,a,b\in V.\end{equation}
  \end{pr}
 Let $\Ku$ be a semisymmetric curvature tensor on $n-$dimensional pseudo-Riemannian vector space  $(V,\prs)$. Then     ${\mathfrak{h} }(\Ku)=\bar{\h}(\Ku)$  is the primitive holonomy Lie algebra of $\Ku$ and  we set  ${\cal H}(\Ku):=\bar{\cal H}(\Ku)$  its the  primitive holonomy Lie group. Moreover, $\Ku$ is invariant by ${\cal H}(\Ku)$ and  it is parallel under the action of ${\mathfrak{h} }(\Ku)$, i.e;
 \[\sigma^*\Ku=\Ku \ssi (\sigma^*\Ku )(a,b)u=\sigma\big(\Ku(\sigma^{-1}(a),\sigma^{-1}(b))(\sigma^{-1}(u))\big)=\Ku(a,b)(u)\] and
  \[\Ku(a,b).\Ku=0 \ssi [\Ku(a,b),\Ku(u,v)]= \Ku(\Ku(a,b)(u),v)+\Ku(u,\Ku(a,b)(v)), \]  for all  $\sigma\in {\cal H}(\Ku)$ and for all $(a, b,u, v)  \in V^4$.

 So the action of the primitive holonomy Lie algebra on the space $V$ introduces a weakly  decomposition of $V$;
 \begin{equation}\label{eqv9}
 V=V_0+ V_1+\ldots+ V_s+ V_0'
 \end{equation}where $V_0=\Big\{ x\in V\;/\; \K(u,v)x=0, \forall (u,v)\in V^2\Big\} $,  $V_0'$ is the dual of the subspace $V_0\cap  (V_1+\ldots+ V_s)$  and each subspace  $V_i+(V_i\cap  V_i^\bot)'$ is indecomposable subspace under the action of ${\mathfrak{h} }(\Ku)$ for all $i\geq1$. 
  \begin{Def} The decomposition (\ref{eqv9}) is called  {\it primitive decompostion} of $V$.

 \end{Def}
Moreover, the   primitive decomposition satisfies the following properties;  
 \begin{enumerate}
	\item for any  $i=0,\ldots,s$, $V_i$ is ${\mathfrak{h} }(\Ku)$-invariant,
	\item for any $i,j=0,\ldots,s$ with $i\not=j$, $\Ku(V_i, V_j)=0$,

\end{enumerate}
\subsubsection{Ricci decomposition}

 \begin{Le}\label{Le1-32} Let $ \Ku $ be a semi-symmetrical curvature tensor on the pseudo-Riemannian space $ (V, \prs) $. Then, its Ricci operator commutes with all the endomorphisms $ \Ku (u, v) $, that is:
\begin{equation}\label{eq1-10} \Ku(u,v)\circ\mathrm{Ric}=\mathrm{Ric}\circ \Ku(u,v),~~~\forall ~u,v\in V.
\end{equation}\end{Le}
\begin{proof} $\;$

Soit $u$, $v$, $w$, $z\in V$.
\begin{eqnarray*}
  \langle \Ri(\Ku(u,v)w,z\rangle &=&\tr(\tau(z,\Ku(u,v)w))\\&=& \tr(t\mapsto \Ku(z,t)(\Ku(u,v)w)) \\
   &=& \tr(t\mapsto \Ku(u,v)(\Ku(z,t)w))-\tr(t\mapsto \Ku(z,\Ku(u,v)t)w)\\
   &&-\tr(t\mapsto \Ku(\Ku(u,v)z,t)w) \\&=&\tr([\Ku(u,v),\tau(z,w)])-\tr(\tau(\Ku(u,v)z,w))\\&=&
   \langle \Ku(u,v)\circ\Ri(w),z\rangle.
\end{eqnarray*}\end{proof}

A Ricci operator $\Ri$ satisfying the equation (\ref{eq1-10}) is said  semi-symmetric. In particular, for a pseudo-Riemannian manifold whose Ricci operator is parallel ($ \na \Ri = 0 $), $ \Ri $ is semi-symmetric. This type of manifolds was studied by C. Boubel (See \cite{boubel}) by giving the following classification theorem:

\begin{theo}\label{tboubel}(\cite{boubel})
	Let $(M,g)$ be a pseudo-Riemannian manifold with a parallel Ricci
	curvature $\Ri$(i.e. $\na.\Ri=0$) and let $\chi$ be the minimal polynomial of $\Ri$. Then:  
   \begin{enumerate}
            \item $\chi=\Pi_iP_i$ where:\begin{enumerate}
            \item [$\bullet$] $\forall  i\neq j$, $P_i\we P_j=1 $ (i.e, $P_i$ et $P_j$ are mutually prime),
     \item [$\bullet$] $\forall  i$, $P_i$ is irreducible or  $P_i=X^2$.
   \end{enumerate}
   \item There is a canonical family $(M_i)_i$ of pseudo-Riemannian manifolds such that the minimal polynomial of $\Ri_i$ on each $M_i$ is $P_i$, and a local isometry $f$ mapping the
   Riemarmian product  $\Pi M_i$ onto $M$. $f$ is unique up to composition with a product of
   isometries of each factor $M_i$. If $M$ is complete and simply connected, $f$ is an isometry.
           \end{enumerate}
\end{theo}
For the proof of the first result of this theorem above, {C. Boubel} only used the following hypothesis: On the tangent space $ T_xM $ at each point $ x \in M $,  the Ricci operator $ \Ri_x $ commutes with the endomorphisms $ R_x (u, v) $ for all $ u $, $ v \in T_xM $. A result which remains valid for spaces provided with semi-symmetrical curvatures and more particularly, for spaces with semi-symmetrical   Ricci. Thus, the following result:

 \begin{theo}\label{pr1-3-1}
 Let $ \Ku $ be a semisymmetric curvature tensor on a pseudo-Riemannian space $ (V, \prs) $  and let $ \chi $ be the minimal polynomial of its Ricci operator $\Ri$. Then the following properties are checked:
   \begin{enumerate}
            \item $\chi=\Pi_iP_i$  where:\begin{enumerate}
            \item [$\bullet$] $\forall  i\neq j$, $P_i\we P_j=1 $,(i.e, $P_i$ et $P_j$ are mutually prime),
            \item [$\bullet$] $\forall  i$, $P_i$ is irreducible or  $P_i=X^2$.
   \end{enumerate} \item $V$ splits orthogonally as \begin{equation}\label{dec-V-Ric} V=E_0\oplus E_1\oplus\ldots\oplus E_r,\end{equation} where $E_0=\ker((\mathrm{Ric}^2))$ and $E_i=\ker(P_i(\mathrm{Ric}))$,
\item for any $u,v\in V$ and $i=0,\ldots,r$, $E_i$ is ${\mathfrak{h} }(\Ku)$-invariant,
\item for any $i,j=0,\ldots,r$ with $i\not=j$, $\Ku_{|E_i\wedge E_j}=0$,
\item for any $i=1,\ldots,r$, $\dim E_i\geq2$.
\item The primitive holonomy algebra $\h(\Ku)$ is satisfied;  $$\h(\Ku)=\h_0(\Ku)+\h_1(\Ku)+...+\h_r(\Ku),$$ 
where $\h_i(R):=\{\Ku(X,Y)\;/\; X,Y\in E_i\}$ is a Lie subalgebra for all $0\leq i\leq r$.
\end{enumerate}

 \end{theo}
 \begin{Def} The decomposition (\ref{dec-V-Ric}) is called   Ricci decompostion of $V$.
	
\end{Def}
\begin{rem}
  If $P_0=X$, we get that $\ker(\mathrm{Ric})=\ker(\mathrm{Ric}^2)$.
  That's why we always set $E_0=\ker((\mathrm{Ric}^2))$.
\end{rem}
 \begin{proof}$\;$

 \begin{enumerate} \item  
In the proof of the theorem \ref{tboubel}, \textbf {C. Boubel} only used the commutativity of the Ricci operator with the each endomorphisms $ R (u,v) $ which remains valid for the first result of our theorem. 

\item Let  $\chi=\Pi_iP_i$ be the irreducible decomposition of $\chi$.

Then, $V$  splits orthogonally as:
\[V=\bigoplus_iE_i,~~ \text{ where }~~ E_i=\ker(P_i(\Ri)).\] 
  \item Let $(u,v)\in V^2$. So $\Ku(u,v)$ commutes  with  $\Ri$, then each  $E_i$ is invariant by $\Ku(u,v)$.
\item Let $u\in E_i$,  $v\in E_j$ and $a,b\in V$ such that $i\neq j$. Since $ \Ku(a,b)(E_i)\subset E_i$ and $\langle E_i,E_j\rangle=0$, we obtain
\[ 0=\langle  \Ku(a,b)u,v\rangle\stackrel{\eqref{eq1-5}}=\langle  \Ku(u,v)a,b\rangle \]  and then, $ \Ku(u,v)=0$.
\item We assume that it exists $i\in \{ 1,\ldots,r\}$ such that $\dim E_i=1$. We choose an element $e$ in $E_i$  checking $\langle e,e\rangle=\e$ with $\e^2=1$ and we complete it to an orthonormal basis $(e,e_1,\ldots,e_{n-1})$ with $\langle e_i,e_i\rangle=\e_i$ and  $\e_i^2=1$. For all $a,b\in V$, we get that $ \Ku(a,b)$ is skewsymmetric endomorphism which leaves $ E_i $ invariant. Then $ \Ku(a,b)e=0$, while
\[ \e\al_i=\langle \mathrm{Ric}(e),e\rangle=\e\langle \Ku(e,e)e,e\rangle+\sum_{i=1}^{n-1}\e_i\langle  \Ku(e,e_i)e,e_i\rangle=0, \]
which is absurd and this completes the proof of the theorem.
 \qedhere
\end{enumerate}

\end{proof}

This proposition reduces  the determination of semi-symmetric curvature tensors on pseudo-riemannian vector spaces to the determination of two classes of semi-symmetric curvature tensors:
 \begin{enumerate}
    \item Einstein semi-symmetric curvature tensors  ( with  complexification in the case of non-real complex eigenvalues)
    \item  Semi-symmetric curvature tensors with  Ricci operator  is satisfied   $\Ri^2=0$.
                                                       \end{enumerate}
\subsection{ Semi-symmetrical Lorentzian space}

\begin{pr}\label{pr1}(\cite{benromane})
Let $\Ku$ be a semi-symmetric curvature tensor on a Lorentzian vector space $(V,\prs)$. Then all eigenvalues of $\mathrm{Ric}_K$ are real. Denoted by $\al_1,\ldots,\al_r$ the non null eigenvalues and $E_1,\ldots,E_r$ the corresponding eigenspaces. Then:
\begin{enumerate}\item $V$ splits orthogonally as $V=E_0\oplus E_1\oplus\ldots\oplus E_r$, where $E_0=\ker(\mathrm{Ric}^2)$,
	\item for any  $i=0,\ldots,r$, $E_i$ is ${\mathfrak{h} }(\Ku)$-invariant,
	\item for any $i,j=0,\ldots,r$ with $i\not=j$, $\Ku_{|E_i\wedge E_j}=0$,
	\item for any $i=1,\ldots,r$, $\dim E_i\geq2$.
	
\end{enumerate}
 Moreover, the primitive holonomy algebra $\h(\Ku)$ splits orthogonally:  $$\h(\Ku)=\h_0(\Ku)+\h_1(\Ku)+...+\h_r(\Ku),$$ 
	where $\h_i(\Ku):=\{\Ku(u,v)\;/\; u,v\in E_i\}$ is a Lie  subalgebra  for all $0\leq i\leq r$.\\
	More precisely,  one of the following two situations occurs: \begin{enumerate}
		\item[a)] Ricci is diagonlizable; $E_0=\ker(\Ri)$.
	\item[b)] Ricci is of   isotropic type: $\ker(\Ri)\subsetneq E_0=\ker(\Ri^2)$.
		\end{enumerate} 
\end{pr}

\begin{proof}This is a special case of theorem.\ref{pr1-3-1}
and it suffices to show that Ricci has only real eigenvalues, this is equivalent to showing that Ricci's minimal polynomial has no non-real roots.
	
Suppose that $ \Ri $ admits a non-real eigenvalue $ z = a + ib $ where $ b \neq0 $. Then, $ z $ and $ \bar{z} $ are roots of an irreducible factor of degree $ 2 $ of the minimal Ricci polynomial.\\
 Let $e$ and $\overline{e}$ in $V$ such that \[\langle e,e\rangle=-\langle\overline{e},\overline{e}\rangle=1,\; \langle {e},\overline{e}\rangle=0,\; \mathrm{Ric}(e)=ae-b\overline{e}\esp
	\mathrm{Ric}(\overline{e})=be+a\overline{e}.  \]
	Then, the subspace $E=\mathrm{span}\{ e, \overline{e}\}$ and its orthogonal subspace $E^\perp$ are  stable by $\Ri$ and by  endomorphisms $\Ku(u,v)$ for all $u$ and $v$ in $V$. Let   $\B=(e_1,\ldots,e_{n-2} )$ be a   orthonormally basis of  $E^\perp$.

	Then, we have:
	\begin{eqnarray*}
		b&=&\langle \mathrm{Ric}(\overline{e}),e\rangle\\
		&=&\langle \Ku(\overline{e},e)e,e\rangle-\langle \Ku(\overline{e},\overline{e})e,\overline{e}\rangle+\sum_{i=1}^{n-2}\langle \Ku(\overline{e},e_i)e,e_i\rangle\\&=&0,
	\end{eqnarray*} which contradicts the fact that $b\not=0$.
\end{proof}

\begin{rem}
	In \cite{benromane},  the author had demonstrated this proposition with different techniques.  
\end{rem}
\begin{Le}\label{lema24}
	Let $\Ku$ be a semi-symmetric curvature tensor on the Lorentzian vector space $(V,\prs)$. Then the prmitive decomposition of $V$ is given as; $$V=\Big(V_0+ V_1+V_0'\Big)\oplus V_2\oplus\ldots\oplus V_s,$$ where $V_0'$ is the dual subspace of $V_0\cap V_1$.
	 
		Therefore,
	  $V_i$ is a Riemannian subspace of dimension greater than or equal to $2$ for any  $i\geq 2$.\\	  
	 And so, 
	 \begin{enumerate}
	 	\item If $E_0$ is a riemannian space, then $V_0=E_0$ is also riemannian space and for  all $i\geq 1$, the   $V_i$ is Einstein space. 
	 		\item If $\Ku$ has a tensor Ricci   of isotropic type, i;e,  $E_0=\ker(\Ri^2)$,  then $E_0=V_0+V_1+V_0'$ such that  $V_0'$ is of dimension $1$ and for  all $i\geq 2$, the   $V_i$ is Einstein Riemannain space.
	 	\end{enumerate}
	 	   
\end{Le}
 \section{Semi-symmetric pseudo-Riemannian differential manifolds}

Let $ (M, g) $ be a pseudo-Riemannian differential manifold of dimension $ n $. Let $ \na $, $ R $, $ \mathfrak{ric} $ and $ \Ri $ are respectively, the Levi-Civita connection, the Riemann curvature, the tensor and the Ricci operator associated to the metric $ g $. We set $ {\cal X}(M) $ the set of vector fields on $ M $. 

So  $(M,g)$   is semisymmetric manifold. Then, at each point
 $m\in M$, the restriction $ R_m $ of $ R $ on $ T_mM $ is a semi-symmetric curvature tensor  and the minimal polynomial of $ \Ri_m $ is of the form $ \X = \prod_{i = 0}^r P_i $ where  the polynomials $ (P_i) _i $ are mutually prime  and for any $ i $, $ P_i $ is irreducible or $ P_i = X^2 $. We can assume that $ P_i $ is irreducible for all $ i \geq1 $. We define the distributions:
  \[E_0(m):=\ker(\Ri_m^2)~~\text{ et}  ~~E_i(m):=\ker(P_i(\Ri_m))~~\text{ for }~~i\geq 1,\]
 and we have the following proposition:

     \begin{pr}\label{pr121} The distributions $(E_i)_i$  checked the following properties:

     For all, $i$, $j\geq 1$ whit $i\neq j$, we have:
     \begin{equation}\label{l}
             \na_{E_j}E_i\subset E_i, \;  \na_{E_i}E_i\subset E_0+E_i, \;
             \na_{E_0}E_i\subset E_i, \;
              \na_{E_0}E_0\subset E_0, \;  \na_{E_i}E_0\subset E_0+E_i.
         \end{equation}

     \end{pr}

    \begin{proof}$\;$

Let $m\in M$. The theorem.\ref{pr1-3-1} induces that: \begin{equation}\label{dec-Tm-ri}  T_mM=E_0(m)\oplus E_1(m)\oplus...\oplus E_r(m). \end{equation}

In the first step, we will show that for $i\geq1$ and $X\in E_i^{\perp}$,   $\na_XE_i\subset E_i$;\\
 Let  $i\geq1$  and  $X\in E_i^{\perp}$. First, we show  $\na_X(R(E_i,E_i)E_i)\subset E_i$:

  Let us take  $Y,\; Z,\; T\in E_i$,  by the second Bianchi identity, we get

				\begin{eqnarray*} \na_X R(Y,Z,T)&:=&(\na_X R)(Y,Z)T\\
&=&-\na_Y R(Z,X,T)-\na_Z R(X,Y,T)\\
   &=&-\na_Y( R(Z,X)T)+ R(\na_YZ,X)T+ R(Z,\na_YX)T+ R(Z,X)\na_YT\\&&-\na_Z( R(Y,X)T)+ R(\na_ZY,X)T+ R(Y,\na_ZX)T+ R(Y,X)\na_ZT\\
   &=& R(\na_YZ,X)T+ R(Z,\na_YX)T+ R(\na_ZY,X)T+ R(Y,\na_ZX)T.
	\end{eqnarray*}
	
	By theorem.\ref{pr1-3-1}, we get; $R(V,V)(E_i)\subset E_i$ and  $\na_X R(Y, Z,T)\in
E_i$.

On the other hand,	
	
         \begin{eqnarray*} \na_X R(Y,Z,T)&=&\na_X( R(Y,Z)T)- R(\na_XY,Z)T- R(Y,\na_XZ)T- R(Y,Z)\na_XT\\
				&=&\na_X( R(Y,Z)T)- R(\na_XY,Z)T- R(Y,\na_XZ)T
\\&&+ R(Z,\na_XT)Y+ R(\na_XT,Y)Z. \end{eqnarray*}
				which proves that,  $\na_X(R(Y, Z)T)\in E_i$.

 Now we will show that $\na_X\Ri(Y)\in E_i$.\\
  We choose an orthogonal basis $(e_1,..., e_n)$  associeted to the decomposition(\ref{dec-Tm-ri}) where $\epsilon_k =\langle e_k,e_k\rangle$  such that $\epsilon_k^2=1$.\\
Let $U \in E_i^{\perp}$.
 If  $e_k \in E_i$, we have already seen that $\na_X(K(Y, e_k)e_k)\in E_i$  and if  $e_k \in E_i^{\perp}$,
 we get   $ R(Y, e_k) = 0$. Consequently, we have
           \begin{eqnarray*}
           \langle \na_X(\Ri(Y)),U\rangle&=& -\langle \Ri(Y), \na_XU\rangle\\
            &=&\sum_{k=1}^n\epsilon_k \langle R(Y,e_k)e_k,\na_XU\rangle\\&=&-\sum_{k=1}^n \epsilon_k \langle \na_X( R(Y,e_k)e_k),U\rangle
   \\&=&0. \end{eqnarray*}
   then $\na_X\big(\Ri(Y) \big) \in E_i$.
   
 If $P_i(t)=t^2+at+b$ whit $b\neq0$, then for all $Y\in E_i$, we have $Y=-\frac{1}{b}(\Ri^2(Y)+a\Ri(Y))$.\\ Consequently, $ \na_XY  \in E_i.$
 
 If $P_i(t)=t-\la_i$  whit $\la_i\neq0$, then for all $Y\in E_i$, we have $Y=\frac{1}{\la_i}\Ri(Y)~~.$\\ Consequently $\na_XY  \in E_i.$

So,  $\na_XE_i \subset E_i$, which shows that $\na_{E_j}E_i \subset E_i$ and   $\na_{E_0}E_i \subset E_i$,
 for all  $i,\; j  \geq1$ whit $i\neq  j$.

 The other results are obtained immediately because the metric $ g $ is parallel(i,e. $\nabla g=0$).

    \end{proof}
    \begin{co}{\label{co31}}
    	Let $ (M, g) $ be a connected semi-symmetric pseudo-Riemannian manifold. Let $\X=\prod_i P_i$ be the minimal polynomial of $\Ri$.  If we set $E_0:=\ker(\Ri^2)$  and for all $i\geq 1$, $E_i:=\ker(P_i(\Ri))$. Then, for all $i\geq1$, the distributions $E_0$ et $E_0+E_i$ are involutive. \\We set ${\cal N}_0$ the integral submanifolds of $E_0$.
    \end{co}

    \begin{rem}
    The distributions $E_0$ and $E_0+E_i$ are involutive but not   necessarly parallele.
   \end{rem} 

On the other way, for a semisymmetrical pseudo-Riemannian manifolds $(M,g)$,  the action of the primitive holonomy Lie algebra $\h(R_m)$  gives the  primitive  decomposition of tangent space $T_mM$ on point $m\in M$ as 
\begin{equation}\label{eqv14}
T_mM=V_0(m)+ V_1(m)+\ldots+ V_s(m)+ V_0'(m)
\end{equation} where
 $V_0(m)=\Big\{ x\in T_mM\;/\; R(u,v)x=0, \;\text{ for all}\; (u,v)\in T_m^2M\Big\} $ and the space $V_0'(m)$ is the dual subspace of $V_0(m)\cap  \{V_1(m)+\ldots+ V_s(m)\}$.

With a similar proof of the proposition\ref{pr121}, we check that the distributions $V_i$ have the following properties;
   \begin{pr}\label{pr32} 
 	
 	For all, $i$, $j\geq 1$ whit $i\neq j$;
 	\begin{equation}\label{l15}
 	\na_{V_j}V_i\subset V_i, \;  \na_{V_i}V_i\subset V_0+V_i, \;
 	\na_{V_0}V_i\subset V_i, \;
 	\na_{V_0}V_0\subset V_0, \;  \na_{V_i}V_0\subset V_0+V_i.
 	\end{equation}
 	
 \end{pr}
   
	Now we go back to the Ricci decomposition and for all $1\leq i\leq r$,
            we consider the distributions  $F_i$  spanned by the vector fields of the forms
             \begin{equation}
              X_1,\nabla_{X_1}X_2,\nabla_{X_1}\nabla_{X_2}X_{3},...,\nabla_{X_1}...\nabla_{X_l}X_{l+1},...\text{ect}.,
             \end{equation} where the vector fields $(X_i)_{i\geq1}$ are  belong to $E_i$.

             In  the same way,  let us the subspaces; $F$, $\tilde{F}_i$,  and  ${F}_{0}$ checking:
             \begin{eqnarray*}             
             F&=&F_{1}+F_{2}+...+F_{r},\\
             TM&=&F\oplus(F\cap F^\bot)'\oplus {F}_{0},\\ 
             \tilde{F}_{i}&=&=F_i\oplus\left(F_i\cap F_i^\bot\right)',\\ 
             \end{eqnarray*} 
	where $\left(  F\cap F^\bot\right)'$ and 
	$\left(F_i\cap F_i^\bot\right)'$ are respectively, the dual subspaces of $ F\cap F^\bot$ and $F_i\cap F_i^\bot$.
\begin{rem}
\item it's obvious that: $$ \forall i\geq1,\;\:F_{0}\cup \left(  F\cap F^\bot\right)\cup \left(  F\cap F^\bot\right)'\subseteq E_0\;\esp \; E_i\subseteq F_i\subseteq E_0+E_i.$$

\end{rem}
 
 For  $i$, $j\geq1$, we put
  $X_1,...,X_k, (\text{resp.}
  Y_1,...,Y_l...,\text{ect},)$ the vector fields belong to  $E_i$ (resp. $E_j$).\\
   In the first step,  we show the following  lemma:
    \begin{Le}  For  all $\:0\leq i\neq j \neq
    	0$,
     the vector fields of  the  forms  $$
      \nabla_{X_1}\nabla_{Y_1}\nabla_{Y_2}...\nabla_{Y_l}Y_{l+1} $$
     are tangent to $F_{j}$, i.e, we have $$\nabla_{E_i}F_{j}\subseteq
     F_{j}.$$
 As a consequence, for all $j \geq 1 $, the distribution $ F_{j}$  is  parallel and the  integral submanifolds are totally geodesic.
    \end{Le}
\begin{proof}
    We can prove this lemma by induction:\\Let $0\leq i\neq j\geq 1$. So $R(X_1,Y_1)Y_2=0$, then $$
      \nabla_{X_1}\nabla_{Y_1}{Y_2}=\nabla_{Y_1}\nabla_{X_1}{Y_2}+\nabla_{[X_1,Y_1]}{Y_2}=\nabla_{Y_1}{Y_2}^*+\nabla_{Y_1^*}{Y_2}-\nabla_{X_1^*}{Y_2}, $$
      where the vector fields ${Y_2}^*:=\nabla_{X_1}{Y_2}$ and
      ${Y_1}^*:=\nabla_{X_1}{Y_1}$  are  tangent to $E_j$ and
      the vector fields ${X_1}^*:=\nabla_{Y_1}{X_1}$ is tangent  to
      $E_i$ for $i\geq 1$ and for $i=0$, we get ${X_1}^*\in E_0+E_j$ and $\nabla_{X_1^*}{Y_2}$ belong to $F_j$. This assumes that the three terms above are tangent to $F_{j}$. So 
         $\nabla_{X_1}\nabla_{Y_1}{Y_2} \in F_{j}$ is checked.\\
        Now, in the genral  case:

        So  $R(X_1,Y_1)=0$, we obtain;
 \begin{eqnarray*} \nabla_{X_1}\nabla_{Y_1}...\nabla_{Y_k}Y_{k+1}&=&\nabla_{Y_1}\nabla_{X_1}\nabla_{Y_2}....\nabla_{Y_k}Y_{k+1}+\nabla_{[X_1,Y_1]}\nabla_{Y_2}...\nabla_{Y_k}Y_{k+1}\\&=&\nabla_{Y_1}\nabla_{X_1}\nabla_{Y_2}....\nabla_{Y_k}Y_{k+1}+\nabla_{{Y_1}^*}\nabla_{Y_2}...\nabla_{Y_k}Y_{k+1}-\nabla_{{X_1}^*}\nabla_{Y_2}...\nabla_{Y_k}Y_{k+1}. \end{eqnarray*}
 As the vector field ${Y_1}^*:=\nabla_{X_1}{Y_1}$ is tangent to $E_j$, and the vector field  ${X_1}^*:=\nabla_{Y_1}{X_1}$ is tangent to
      $E_i$, by the induction hypothesis, we get that $
      \nabla_{X_1}\nabla_{Y_1}...\nabla_{Y_k}Y_{k+1}$ is tangent to
      $F_{j}$. \end{proof}

\begin{Le}\label{lemme201}
	The subspaces $F_{1}$, $F_2$,...and  $F_{r}$  are pairwise orthogonal.
\end{Le}
      \begin{proof} By induction, we show that the vector field
      $\nabla_{Y_1}\nabla_{Y_2}...\nabla_{Y_k}Y_{k+1},\;\;
      k\geq0$ is  orthogonal to $E_i$.\\      
             For $k=0$  and $k=1$, it's obvious.

      Now, suppos that   result  is  true for any vector fields of the 
   form
      $\nabla_{Y_2}...\nabla_{Y_k}Y_{k+1}$. Let 
       $X$ be a vector fields in $E_i$. Then, the vector fields 
      $\nabla_{Y_1}{X}$ is  also tangent to $E_i$. By induction, we get  
      $$g(X,\nabla_{Y_1}\nabla_{Y_2}...\nabla_{Y_k}Y_{k+1})=-g(\nabla_{Y_1}{X},\nabla_{Y_2}...\nabla_{Y_k}Y_{k+1})=0.$$ This  proves the statement.

    Using induction again, we can  prove that the vector fields on the form
      $\nabla_{X_1}\nabla_{X_2}...\nabla_{X_k}X_{k+1}$ are orthogonal to those which are of the form $\nabla_{Y_1}\nabla_{Y_2}...\nabla_{Y_l}Y_{l+1}$.

  For the case $k=0$, the proof is given above. Now, for the vectors fields on the form
  $\nabla_{X_1}\nabla_{X_2}...\nabla_{X_k}X_{k+1}$ are orthogonal
 to  the vector fields on the form
  $\nabla_{Y_1}\nabla_{Y_2}...\nabla_{Y_l}Y_{l+1}$, so the induction
  hypothesis and  lemma(\ref{lemme201}), we
get
  \begin{eqnarray*}g(\nabla_{X_1}\nabla_{X_2}...\nabla_{X_k}X_{k+1},\nabla_{Y_1}\nabla_{Y_2}...\nabla_{Y_l}Y_{l+1})&=&-
  g(\nabla_{X_2}...\nabla_{X_k}X_{k+1},\nabla_{X_1}\nabla_{Y_1}\nabla_{Y_2}...\nabla_{Y_l}Y_{l+1})\\
  &=&0,\end{eqnarray*}
   whil  the vector fields on the form
   $\nabla_{X_1}\nabla_{Y_1}\nabla_{Y_2}...\nabla_{Y_l}Y_{l+1}$
   are tangent to $F_j$. This gives completely the proof of lemma.
\end{proof}

      \begin{co}\label{lemma33}
       $(F\cap F^\bot)'\oplus {F}_{0}$ and $ 
      \tilde{F}_{j}=F_j\oplus\left(F_j^\bot\cap F_j\right)'$ are involutive distributions, for all $j\geq 1$,
      \end{co}
    \begin{proof} We have already seen the relations 
  $\nabla_{E_i}{F_{j}}\subseteq F_{j}$ only for the cases $i\geq 0$ and $j\geq1$. It  remains to show  the  relationships;
   $\nabla_{E_j}{F_{0}}\subseteq F_{0}$  for
  $j> 0$.  The  first is obvious, since
  $$
  g(\nabla_{E_i}{F_{0}},F_k)=-g(F_{0},\nabla_{E_j}F_k)=-g(F_{0},F_k)=0,
  \text{for all} \; k>0.$$

  Finally, for the formula $\nabla_{E_0}{F_{0}}$, we get $$
  g(\nabla_{E_0}{F_{0}},F_{i})=-g(F_{0},\nabla_{E_0}F_{i})=g(F_{0},F_{i})=0.$$
  \end{proof}

\begin{rem}
The distribution  $F_{j}\oplus\left(   F_{j}\cap ( F_{j})^\bot\right)'$ is involutive  non degenerated. 
\end{rem}
\begin{pr}\label{lemme43}
	Let $(M,g)$ be a semi-symmetric, locally connected pseudo-Riemannian manifold. Then ,$ M $ is a foliated manifold and the minimal polynomial of the restriction of the Ricci operator on each leaf has one of the following forms $ X $, $ X ^ 2 $, $P$, $ XP $ or $ X ^ 2P $ where $ P $ is an irreducible  polynomial and it is prime with the polynomial $ X $.
	
In the case where the tangent bundle of $M$ admits only one involutive subbundle nondegenerate, $M$ is called  a simple leaf.
\end{pr}

\begin{Le}\label{lemme43}Let $ (M, g) $ be a simple leaf semisymmetric, locally connected,  pseudo-Riemannian manifold. Then  its a   tangent bundle  has one of the following forms: \begin{enumerate}\item
		$TM= {F_1}\oplus  \left( F_1\cap  \left( F_1\right)^{\bot} \right)'= E_0\overset{\bot}{\oplus} E_1$.
		\item $TM=E_0$.
	\end{enumerate}
\end{Le}



\section{The simple leaf semi-symmetrical Lorentzian manifold}

In this section we will give some properties of semi-symmetric Lorentzian spaces 
\subsection{Ricci decomposition}
 
 In the Lorentzian case, the comparison of primitive and Ricci decompositions produces the following  Proposition; 

\begin{pr}
	\label{proposition41}Let $ (M, g) $ be a simple leaf semisymmetric, locally connected,  Lorentzian manifolds. Then, we get one of the following situations: \begin{enumerate}
			\item $E_0=V_0$, i.e, ${\cal N}_0$ is a plat submanifolds.
		\item The tangent bundle over $M$  is on the forme   $E_0=  V_0+V_1+ (V_0\cap V_1)',$ where $(V_0\cap V_1)'$ is the dual subspace of $(V_0\cap V_1)$.		
		In this situation, we get one of the two following cases: \begin{enumerate}
			\item $E_0=  V_0\oplus V_1$,
			\item  There exists an istropic vector fields $p$ with its dual vector fields $q$ such that
		$E_0=  V_0+V_1+span\{q\}$. In this case, the tensor curvature  holds $$R^2=0.$$
	\end{enumerate}	
	\end{enumerate}
\end{pr}
\begin{proof}
	The   primitive and Ricci decompositions imply that   $$V_0=E_0\;\text{ or}\; V_0\varsubsetneq E_0.$$
	
	If $V_0\varsubsetneq E_0$, this means  that the curvature tensor is non-zero on ${\cal N}_0$ and  the action of the primitive holonomy algebra $ \h_0 (R) $ on $ E_0 $ induces a non-trivial  decomposition:
	$$E_0=  V_0+V_1+ (V_0\cap V_1)',,$$ where $V_{0}:\{x\in E_0/\; \forall h\in\h_0(R),\; h(x)=0\}, V_{1}:=\h_0(R)( E_0)\neq \{0\}$ and $(V_0\cap V_1)'$ is the dual subspace of $ V_0\cap V_1$. 
	
	By   proposition\ref{pr32}, we show that the distribution $  V_{0} + V_{1} $ is involutive  and since $ M $ is a simple leaf, we get that $TM=E_0$.
	
	So, if $V_0\cap V_1\neq \{0\}$ is  a non-trivial subspace, then it is generated by an  istropic parallel vector fields $p$  and  according to the classification of weakly irreducible holonomy algebras of Lorentzian manifolds given by L. B. Bergery and A. Ikemakhen in \cite{ikemakhen}, for any point $ m \in M $, $ \; \h_0 (R_m) $ is a Lie subalgebra of   type  $ 2 $ or $4$. Consequently, each element $ \Ku \in \h_0 (R_m) $ is written in the form
	$$\Ku = \Ku_0-p_m\wedge X,$$
	where 
$\; X\in  { \widehat{E}}_0(m)\cong {E_0(m)}_{/\{p,q\}}$ and  $\Ku_0\in \mathfrak{so}({\widehat{E}_0(m)})\cong\we^2\R^r$, such that $r=\dim(E_0) -2$ and $q$ is an isotropic vector fields which $g(p,q)=1$,  .   
	
	Then the restriction  of the  curvature on the space $E_0(m)$ checks  
	$$\forall (X,Y)\in E_0^2(m),\;  \exists\; E\in { \widehat{E}}_0(m)\; \text{ such that }\;  R_m(X,Y)=\widehat{R}(\widehat{X},\widehat{Y}) -p_m\we E,$$
	 where
	$\widehat{X}$ and $\widehat{Y}$ are the projection respectively  of the $X$ and $Y$ on the subspace $ { \widehat{E}}_0(m)$ and   
	$\widehat{R}$  is the restrection of $R_m$ on $ { \widehat{E}}_0(m)$ and so  $\widehat{R}$ is a semisymmetric  curvature tensor on the irreducible Riemannian space $ { \widehat{E}}_0(m)$ where $\Ri_m=0$. Then $$\widehat{R}=0 \esp R_m(X,Y)= -p_m\we E.$$

thus $$R^2=0.$$
\end{proof}

\begin{pr}\label{pr5vrsi}Let $ (M, g) $ be a connected   simple leaf semi-symmetric Lorentzian manifold. Then the  Ricci operator admits  at most an one  non-zero real eigenvalue. If  $\la$ is such an eigenvalue, then
	the tangent space of $ M $ at any point $ m \in M $ will be of the form:
	\begin{equation}\label{v-dec-irre-l} T_mM=
	\ker\bigg(\Ri_m-\la(m)Id_{T_mM}\bigg)\oplus V_0(m),\end{equation}
	  	   	   Moreover, if $\dim(E_1)\geq3$, then the function $m \in M\val \la(m)$ is of  class $C^\infty$ and it  depends only on  ${\cal N}_0$ the flat  integral submanifolds of $V_0=E_0$.

\end{pr}

\begin{proof}$\;$

Let $ (M, g) $ be a connected, semi-symmetric, simple leaf Lorentzian manifold. Let $ \la $ be a non-zero eigenvalue of $ \Ri $. 

The equation\ref{v-dec-irre-l}  is a resultat of Lemma\ref{lemma33} and Proposition\ref{proposition41}.

Now, we consider the codifferential $ \delta $ on $ M $ given by:
\[\delta (\al) (Y_1, ..., Y_r) = - \Sigma_i (\na_{X_i} \al) (X_i, Y_1, ..., Y_r), \] where
$ \al $ is a $ (r + 1)-$differential form, $ (X_1, ..., X_n) $ is an orthonormal fram on $ M $ and $ (Y_1, ..., Y_r ) $ a family of $ r $ vector fields (See \cite {besse}, page 34).

The Ricci tensor verifies:
	
	\begin{equation}\label{besse(ds)}
	\delta(\mathfrak{ric})=-\frac{1}{2}d (\mathfrak{s}),
	\end{equation}
	where $\mathfrak{s}$ is the scalar curvature, $ d $ is the exterior differential over $ M $ (See proposition 1.94 \cite{besse} page 43).
	
	Let $(X_1,...,X_n)$ be a   orthonormal fram of $TM$ whit  the $X_i$ are tangent to $E_0$ for $i\leq\nu=\dim(E_0)$ and for $i> \nu$, the $X_i$ are tangent to $E_1$ .
	
	Let us $j> \nu$.  we  get:
	\begin{eqnarray*}
		0&=&\frac{1}{2}d(\mathfrak{s})(X_j)+\delta(\mathfrak{ric})(X_j)\\
		&=&\frac{1}{2}X_j(\mathfrak{s})-\Sigma_i(\na_{X_i}\mathfrak{ric})(X_i,X_j)
		\\
		&=&\frac{1}{2}X_j((n-\nu) \la)-\Sigma_i {X_i}.(\mathfrak{ric}(X_i,X_j))-\mathfrak{ric}(\na_{X_i}X_i,X_j)
		-\mathfrak{ric}(X_i,\na_{X_i}X_j),  \\
		&=&\frac{n-\nu}{2}X_j(\la)-\Sigma_i  ({X_i}.\la)g(X_i,X_j) +\la{X_i}.g(X_i,X_j)-\mathfrak{ric}(\na_{X_i}X_i,X_j)
		-\mathfrak{ric}(X_i,\na_{X_i}X_j),  \\
		&=&(\frac{n-\nu}{2}-1)X_j(\la). \\
	\end{eqnarray*}
	
	So  $n-\nu\geq3$  implies that $ X_j (\la) = 0 $ and therefore, $ \la $ only depends on the integral submanifolds  of $ E_0 $.
	
	\end{proof}



\subsection{Basic formulas}

	In this subsection, we 	 
  consider $ (M, g) $   a simple leaf semi-symmetric Lorentzian manifold of dimension $ n   $ such that the Ricci   admits a non-zero eigenvalue $ \la $ of multiplicity $n-r$ such that  $E_0$ is a Lorantzian subspace.
  
 So at any point $ m \in M $, the tangent space splits as:
\[ T_mM=E_0(m)\oplus E_1(m),\]
where \[E_0(m)=\ker(\Ri_m) ~\text{ and } ~E_1(m)=\ker(\Ri_m-\la(m)id_{T_mM}).\]
and ${\cal N}_0$ the integral submanifold of $ E_ {0} $   is   flat.

The real number  $r  =\dim(E_0(m))$ called the nullity index  of the curvature at the point $ m $. The  multiplicity  of eingenvalue $\la$ also called the co-nullity index. 

So ${\cal N}_0$ the integral submanifold of $ E_ {0} $   is   flat,   then we can shoose 
 a fram
 $(e_1,...,e_r)$  of $E_{0}$ such that,  for all $i\geq2$,    
 $\varepsilon_i= g(e_i,e_i)=-g(e_1,e_1)=1,\esp  g(e_i,e_j)=0$,  for all $ ~~1\leq i\neq j\leq r.$
 
  Let us $(u_1,...,u_r)$
  a local coordinate system associated to  $(e_1,...,e_r)$, i.e.   $e_i=\frac{\partial}{\partial u_i}$ for all $1\leq i\leq r$.
  
   On the other hand, $R(E_0,E_0)E_1=0$, then
 we can also choose  $(X_1,...,X_{n-r})$   an orthonormal fram of  $E_1$ such that, for all $1\leq j\leq n-r$, $X_j$ is parralel on ${\cal N}_0$, i.e. $\nabla_{e_{i}}X_{j}=0$.

For all vector fields $e_i$ and for all vector fields  $X$, we set
\begin{equation}
\nabla_{X}e_i=A_{i}(X)+\sum_{j=1}^{r}B_{i}^j(X)e_j,
\end{equation}
where $A_{i}(X)$  is the orthogonal projection of $\nabla_{X}e_i$ on
$E_{1}$.

$A_{i}$  is $(1,1)-$tensor   on  $M$
which is zero on $ E_{0} $ and   $ B_{i}^j $ are  covariant tensors on $ M $ wich they  have the value zero  on $ E_{0} $.

Moreover, we get

\begin{equation}
B_{i}^j(X)=\varepsilon_jg(\nabla_{X}e_i,e_j)=-\varepsilon_jg(e_i,\nabla_{X}e_j)=-\varepsilon_j\varepsilon_iB^{i}_j(X).
\end{equation}
\begin{Def}
The field tensors $ A_i $ and $ B_{i}^j $ are called the
second fundamental forms corresponding to the
system  $\{e_1,...,e_r\}$.
\end{Def}

We define the $(0,2)$-tensor $B^i$   by:
\begin{equation}\label{M(0,2)forme}
\begin{array}{ccccc}
B^i(X,Y) &:=& -g(A_i(X),Y),& &\\
B^i(X,e_j) &:=&B^i(e_j,X)&=&B^i(e_i,e_l)=0,
\end{array}
\end{equation}
where $X$ and $Y$ are  vectors fields tangent to $E_1$.
\begin{Le}\label{lemma41}	
	The second fundamental forms $A_i$  and $B_j^i$ and the curvature $ R $ satisfy the following properties:
	\begin{equation}\label{eq22}
	2trace(A_i)=-\frac{(n-r)}{\la}\frac{\partial \la}{\partial u_i},~~~~1\leq i\leq r,
	\end{equation}
	\begin{equation}\label{naB1}
\na_{e_i}A_j(X)=-A_j\circ A_{i}(X),
\end{equation}
	\begin{equation}\label{naB2}
(\na_{e_i}B_j^k)(X)=-B_j^k(A_i(X)),
\end{equation}

	\begin{equation}\label{nabRRoBxy}
	(\na_{e_i}R)(X,Y)=R(Y,A_{i}(X))+R(A_{i}(Y),X),
	\end{equation}
	
	\begin{equation}\label{BianBB}
	R(X,Y)A_{i}(Z)+R(Y,Z)A_{i}(X)+R(Z,X)A_{i}(Y)=0,
	\end{equation}
	for $X$, $Y$, $Z \in  E_{1}$. \end{Le}

\begin{proof}
		
	The formula (\ref{eq22}) is result of	 the formula(\ref{besse(ds)}). Indeed;
		
		\[\delta(\mathfrak{ric})(e_i)=-\frac{1}{2}d (\mathfrak{s})(e_i)\]

	Let us $X$, $Y$, $Z\in E_1$.\\
	Formulas (\ref{naB1}) and (\ref{naB2}) follow from the equation $R(e_i,X)e_j=0$.
	
	The formulas (\ref{nabRRoBxy}) and (\ref{BianBB}) follow  frome the second Bianchi identity, indeed: 
	\begin{eqnarray*}
	(\na_{e_i}R)(X,Y)&=&(\na_XR)(e_i,Y)+(\na_YR)(X,e_i)\\
	&=&[\na_X,R(e_i,Y)]-R(\na_Xe_i,Y)-R(e_i,\na_XY)\\
	&&+[\na_Y,R(X,e_i)]-R(\na_YX,e_i)-R(X,\na_Ye_i)\\
	&=&-R(\na_Xe_i,Y)
-R(X,\na_Ye_i)\\
&=&R(Y,A_i(X))+R(A_i(Y)+X)
	\end{eqnarray*}
So $R(X,Y)(e_i)=R(Z,X)(e_i)=R(Y,Z)(e_i)=0$, 

we get
	\begin{eqnarray*}
	0&=&(\na_{X}R)(Y,Z)(e_i)+(\na_YR)(Z,X)(e_i)+(\na_ZR)(X,Y)(e_i)\\
	&=&-R(Y,Z)(\na_{X}e_i)-R(Z,X)(\na_{Y}e_i)-R(X,Y)(\na_{Z}e_i)
\end{eqnarray*}
\end{proof}

\begin{co}
	$(\nabla_{e_i}R)$ is a curvature tensor on $M$ and  $\nabla_{e_i}R(X,Y) \in \h_1(R)$ for all $X$, $Y\in E_1$ .
\end{co}

\begin{pr}There exists a function $ \mu_i $ of class $ C^{\infty} $ on $ M $ such that \begin{equation}\label{R-nabla-R}
	\nabla_{e_i}R(X,Y)=-2\mu_i R(X,Y),
	\end{equation} where $X$, $Y\in E_1$.
\end{pr}

\begin{proof}$  $
	
	Let $ m $ be a point in $ M $. So the primitive holonomy group $ {\cal H}_1 (R_m) $  acts irreducibly on $ E_1 (m) $. Therefore,  $[E_1(m),R_m,{\cal H}_1(R_m)]$ and $[E_1(m),(\nabla_{e_i}R)_m,{\cal H}_1(R_m)]$ are two  Riemannian irreducible  symmetric holonomy systems. According to the corollary of Theorem.6 in (\cite{simon}), there exists a  real $ \mu_i(m) $ satisfying 
	\[(\nabla_{e_i}R)_m=-2\mu_i(m) R_m.\]
So the tensors $ \nabla_{e_i} R $ and $ R $ are of class $ C^{\infty} $, then  the function $ \mu_i $ is of class $ C^{\infty} $. \end{proof}

\begin{Le}
	Any second fundamental form $ A_i $ on $ E_1 $ has one of the following forms $ A_i = \mu_i I $ or else $ n \in 2\N $ and $ A_i = \mu_i I + \la_i J $ with $J^2=-I$, where  $ \mu_i $ and $ \la_i $ are functions of class $ C^\infty $. The skewsymmetric endomorphism $ J $ is a uniquely determined  and is  independent of $ i $ and  the choice of the system $ (e_1, ..., e_r) $ and commutes with each elements of $ {\cal H}_1(R_m) $ at any point $ m \in M $.
\end{Le}

\begin{proof}
	
The proof is similar as  that of Lemma 4.4 in (\cite{zabo}) considering the Lie algebra ${\cal H}_1(R_m)$ and the Riemannian irreducible symmetric holonomy system $S_m:=[E_1,R_m, {\cal H}_1(R_m)]$.
\end{proof}

From the formula\ref{eq22}, we get 	\begin{equation}
	 \mu_i=-\frac{1}{2\la}\frac{\partial \la}{\partial u_i},~~~~1\leq i\leq r.
	\end{equation}
So, if $A_i=\mu_i I+\la_i J$, since $\na_{e_i}A_i=-A_i^2$, we get 
\begin{equation}
\na_{e_i} J=0.
\end{equation}
\begin{equation}
\na_{e_i} \mu_i=\la_i^2-\mu_i^2.
\end{equation}
 \begin{equation}
 -2\la_i \mu_i=\frac{\partial \la_i}{\partial u_i}.
 \end{equation}

Then \begin{equation}\label{eq4-32}\la\frac{\partial \la_i}{\partial u_i}=\la_i\frac{\partial \la}{\partial u_i}\end{equation}

\begin{Le}\label{le2-33}
		If    the co-nullity index satisfies   $n-r\geq 3$, we get that  \[r\in\{1,2\}.\]
\end{Le}
\begin{proof}$ $

	By lemma\ref{lemme43}, we get  $$TM={F_1}\oplus  \left( F_1\cap  \left( F_1\right)^{\bot} \right)'=E_0\oplus E_1,$$ and   \[F_1=span\{pr_0(\na_XY)/~X,~Y\in E_1\}\]
where  $pr_0~: TM\vals { {E}}_0$   is the orthogonal projection in ${ {E}}_0$. We will show that  \[r=\dim(span\{pr_0(\na_XY)/~X,~Y\in E_1\}).\]
	
	Necessarily, we have  $r\geq1$. Let $(e_1,...,e_r)$ be an orthonormal fram of ${ {E}}_0$.
	
	First, we show $\dim(span\{pr_0(\na_XX)/~X\in E_1\})=1$;\\
	Let us  $X$, $Y\in E_1$. If $X$ and   $Y$ are orthogonal unit field. Let  $A_i=\mu_i I+\la_i J$ be a   second  fundamental form,  where $J$ is skew-symmetric satisfying  $J^2=-I$. Then \[g(\na_XX , e_i)=-\mu_i g(X,X)=-\mu_i g(Y,Y)=g(\na_YY , e_i).\]
	consequently, $pr_0(\na_XX)=pr_0(\na_YY)$.
	
If	  $X$ and   $Y$ are not orthogonal. Since $\dim(E_1)\geq 3$, we can choose $Z$ orthogonal to both  $X$ and  $Y$.
	
	Then, $pr_0(\na_XX)=pr_0(\na_ZZ)=pr_0(\na_YY)$, thus 	
	\begin{equation}\label{eq2-23}
	\dim(span\{pr_0(\na_XX)/~X\in E_1\})=1.\end{equation}
	Moreover, there is no isotropic vector field in $F_1$ unless $\dim(E_0)=2$.
	
	If any second  fundamental forms are of the form $A_i=\mu_i I$, we get that \[\dim(span\{pr_0(\na_XY)/~X,~Y\in E_1\})=1.\]
	Indeed, for all orthogonal fields $X$ and $Y$ in $E_1$, we have \[g(\na_XY , e_i)=-\mu_i g(X,Y)=g(\na_YX , e_i)=0.\] Hence $span\{pr_0(\na_XY)/~X,~Y\in E_1\} =span\{pr_0(\na_XX)/~X \in E_1\}.$

If there is a second fundamental form $A_i=\mu_i I+\la_i J$ where  $\la_i\neq 0$ and $J$ is skewsymmetric endomorphism checking $J^2=-I$.\\
	Then $\dim(E_1)=2l$  and we can choose $\{X_1,Y_1=J(X_1),...,X_l,Y_l=J(X_l)\}$ an  orthonormal fram of  $E_1$.
	
	Let $A_k=\mu_k I+\la_k J$ be the second fundamental form associted   to the vector fields $e_k$. For all $i\neq j$, we get
	\begin{equation}\label{eq2-24} \begin{array}{ccccc}
	g(\na_{ X_i}Y_i,e_k)&= g(\na_{ X_j}Y_j,e_k)&=&-\la_k& \\
	g(\na_{ X_i}Y_j,e_k)&=g(\na_{ Y_j}X_i,e_k)&=&g(\na_{ X_j}X_i,e_k)&= g(\na_{Y_j}Y_i,e_k)=0.\end{array} \end{equation}
	
	Thus $$\dim\bigg(span\bigg\{pr_0(\na_XY)/~X , Y \in E_1, \; g(X,Y)=0\bigg\}\bigg)=1$$
	
	and $F_1=span\{pr_0(\na_XY)/~X , Y \in E_1, \; g(X,Y)=0\}+span\{pr_0(\na_XX)/~X  \in E_1\}.$
	
	Consequently, we get \[r=\dim\bigg(span\bigg\{pr_0(\na_XY)/~X,~Y\in E_1\bigg\}\bigg)\leq 2.\]
	
\end{proof}



	\end{document}